\documentclass[10pt]{amsart}

\usepackage{amssymb, amsmath, amsthm, amsfonts}

\usepackage{mathrsfs}
\usepackage{extarrows}
\usepackage{amsthm}
\usepackage{amssymb}

\usepackage[all]{xy}

\usepackage{verbatim}
\usepackage{hyperref}
\hypersetup{
    colorlinks=true,       
    linkcolor=blue,          
    citecolor=blue,        
    filecolor=blue,      
    urlcolor=blue           
}

\newtheorem{theorem}{Theorem}[section]
\newtheorem{cor}[theorem]{Corollary}
\newtheorem{lem}[theorem]{Lemma}
\newtheorem{prop}[theorem]{Proposition}

\theoremstyle{definition}
\newtheorem{defn}[theorem]{Definition}

\theoremstyle{remark}

\newtheorem{rems}[theorem]{Remarks}

\newcommand{\HH}{\mathrm{H}}

\newcommand{\bZ}{\mathbb{Z}}

\newcommand{\bC}{\mathbb{C}}
\newcommand{\bV}{\mathbb{V}}

\newcommand{\cA}{\mathcal{A}}
\newcommand{\cO}{\mathcal{O}}
\newcommand{\cD}{\mathcal{D}}

\newcommand{\cF}{\mathcal{F}}

\newcommand{\cL}{\mathcal{L}}

\newcommand{\arrow}{\rightarrow}
\newcommand{\rk}{\mathrm{rk}\,}

\newcommand{\Sym}{\mathrm{Sym}}

\newcommand{\sbt}{\,\begin{picture}(-1,1)(-1,-3)\circle*{3}\end{picture}\ }

\title[]{Hyperbolicity of varieties supporting a variation of Hodge structure}
\author{Yohan Brunebarbe, Beno\^{i}t Cadorel}
\date{\today}

\begin{document}

\maketitle

\begin{abstract}
We generalize former results of Zuo and the first author showing some hyperbolicity properties of varieties supporting a variation of Hodge structure. Our proof only uses the special curvature properties of period domains. In particular, in contrast to the former approaches, it does not use any result on the asymptotic behaviour of the Hodge metric.
\end{abstract}

\section{Introduction}

The goal of this note is to give a proof of the following result, which generalizes both \cite[Theorem 0.1]{Zuo_neg} and \cite[Theorem 0.3]{Brunebarbe_Crelle}. Interestingly, whereas the proofs in \textit{op. cit.} use difficult results about degenerations of Hodge structure, the proof given here uses only the special curvature properties of period domains. 

\begin{theorem}\label{main_result}
Let $X$ be a compact complex manifold and $D \subset X$ be a normal crossing divisor. Assume that $X-D$ admits a complex polarized variation of Hodge structure whose period map is immersive at one point. Then
\begin{enumerate}
\item the logarithmic canonical bundle $\omega_X(D) = \det \Omega^1_X(\log D)$ is big, i.e. the variety $X-D$ is of log-general type.
\item the logarithmic cotangent bundle $\Omega^1_X(\log D)$ is weakly-positive and big. 
\end{enumerate}
\end{theorem}

Recall that a line bundle $\cL$ on a compact complex manifold $X$ is called big if there exists a positive constant $c$ such that $h^0(X, \cL^{\otimes k}) \geq c \cdot k^{\dim X}$ for $k \gg 1$. More generally, we will say in this article that a vector bundle $\cA$ on $X$ is big if there exists a positive constant $c$ such that  $h^0(X, \Sym^k \cA) \geq c \cdot k^{\dim X + \rk \cA -1}$ for $k  \gg 1$. Also, following Viehweg \cite[Definition 1.2]{Viehweg83}, we say that a vector bundle $\cA$ on a smooth projective variety $X$ is weakly-positive if there exists a Zariski-dense open subset $U \subset X$ such that for every big line bundle $\cL$ on $X$ and every positive integer $\alpha$, there exists a positive integer $\beta$ such that $\Sym^{\alpha \cdot \beta} \cA \otimes \cL^\beta$ is generated by global sections on $U$. More generally, we say that a vector bundle $\cA$ on a Moishezon compact complex manifold $X$ is weakly-positive if there exists a smooth projective variety $X^\prime$ and a bimeromorphic holomorphic map $\nu : X^\prime \arrow X$ such that $\nu^\ast \cA$ is weakly-positive.

\begin{rems}
\begin{enumerate}
\item Note that, contrarily to \cite{Zuo_neg} and \cite{Brunebarbe_Crelle}, we make no assumption on the monodromy at infinity of the local system underlying the variation of Hodge structure. Also, we don't need $X$ to be projective or K\"ahler.
\item If in addition $X$ is assumed to be projective, then the first assertion is a consequence of the second, since the determinant of a weakly-positive and big vector bundle is big (cf. Lemma \ref{det_nef_and_big}). Alternatively, one can also appeal to a result of Campana and P\u aun \cite{Campana-Paun} which says that a projective log-pair with a big logarithmic cotangent bundle is of log-general type. Since we make no projectivity assumption on $X$, we give a direct proof of $(1)$ without knowing $(2)$. 
\item In the situation of the theorem, the logarithmic cotangent bundle $\Omega^1_X(\log D)$ is not big in the sense of Viehweg in general. For example, the logarithmic cotangent bundle of a locally hermitian symmetric variety not covered by the ball is never Viehweg-big. 
\end{enumerate}
\end{rems}

\begin{cor}
Let $X$ be a compact complex manifold. Assume that on a dense Zariski-open subset of $X$ there exists a complex polarized variation of Hodge structure whose period map is immersive at one point. Then $X$ is Moishezon.
\end{cor}
A fortiori, if in addition $X$ admits a K\"ahler metric, then it is necessarily projective.\\

Observe that it follows readily from Theorem \ref{main_result} that any complex polarized variation of Hodge structure on the complex line $\bC$ is constant, as already proved by Griffiths and Schmid \cite[Corollary 9.4]{Griffiths-Schmid}. 
More generally, we obtain as a direct application of Theorem \ref{main_result} that complex algebraic varieties supporting a variation of Hodge structure satisfy the following analogue of the Green-Griffiths-Lang conjecture:

\begin{cor}
Let $U$ be a smooth complex algebraic variety admitting a complex polarized variation of Hodge structure $\bV$. If $\mathrm{Deg}(\bV) \subset U$ denotes the closed (algebraic) subvariety where the period map is not immersive, then:
\begin{itemize}
\item every entire curve of $U$ is contained in $\mathrm{Deg}(\bV)$;
\item every algebraic subvariety of $U$ which is not of log-general type is contained in $\mathrm{Deg}(\bV)$.
\end{itemize}
\end{cor}

\begin{rems}
\begin{enumerate}
\item The period map being holomorphic, it is clear that $\mathrm{Deg}(\bV)$ is a closed analytic subset of $U$. In fact it turns out to be an algebraic subvariety.
\item When $U$ is defined over the ring of integers of a number field, the Lang-Vojta type conjectures predicts that all but finitely many integer points belong to $\mathrm{Deg}(\bV)$. This follows easily from Faltings results when $\bV$ is a $\bZ$-polarized variation of Hodge structure of weight one, but is open otherwise.
\end{enumerate}
\end{rems}

Once we know the curvature properties of period domains (that are recalled in section \ref{Hodge}), Theorem \ref{main_result} is a consequence of the following result of independent interest:

\begin{theorem}\label{metric_criterion} Let $X$ be a compact complex manifold, $D \subset X$ be a normal crossing divisor and $h$ be a singular hermitian metric with semi-negative curvature on the tangent sheaf $T_X|_{X-D}$ restricted to $X-D$. Assume moreover that in restriction to a dense Zariski-open subset of $X-D$ the metric $h$ is smooth and has negative holomorphic sectional curvature, bounded from above by a negative constant. Then the logarithmic canonical line bundle $\omega_X(D)$ is big (so that $X$ is in fact necessarily Moishezon) and the logarithmic cotangent bundle $\Omega^1_X(\log D)$ is weakly-positive. If moreover the metric $h$ is smooth and K\"ahler in the neighborhood of a point of $X-D$, then the logarithmic cotangent bundle $\Omega^1_X(\log D)$ is big.
\end{theorem}

\section{Proofs}
\subsection{Preliminaries on variations of Hodge structure and curvature properties of period domains}\label{Hodge}

Recall that a complex polarized variation of Hodge structure on a complex manifold $X$ is the data of a holomorphic vector bundle $\mathscr{E}$ equipped with an integrable connection $\nabla$, a $\nabla$-flat non-degenerate hermitian form $h$ and for all $x \in X$ a decomposition of the fibre $\mathscr{E}_x = \bigoplus_{p \in \bZ} \mathscr{E}^p$ satisfying the following axioms:

\begin{itemize}
\item for all $x \in X$, the decomposition $\mathscr{E}_x = \bigoplus_{p \in \bZ} \mathscr{E}_x^p $ is $h_x$-orthogonal and the restriction of $h_x$ to $ \mathscr{E}_x^p$ is positive definite for $p$ even and negative definite for $p$ odd,
\item  the Hodge filtration $\cF^{\sbt}  := \bigoplus_{ q \geq \sbt} \mathscr{E}^q$ varies holomorphically with $x$, 
\item (Griffiths' transversality) $\nabla( \cF^p) \subset \cF^{p-1} \otimes_{\cO_X} \Omega_X^1$ for all $p$.
\end{itemize}  

Fix $x_0 \in X$, and let $E := \mathscr{E}_{x_0}$, $h := h_{x_0}$ and $r_p = \dim(\mathscr{E}_{x_0}^p)$ for every $p \in \bZ$. The associated period domain $\cD$ is by definition the set of all $h$-orthogonal decompositions $E = \bigoplus_{p \in \bZ} E^p$ with $\dim E^p = r_p$ such that the restriction of $h$ to $E^p$ is positive definite for $p$ even and negative definite for $p$ odd. The group $\mathrm{G} := \mathrm{U}(E,h) \simeq  \mathrm{U}(\sum_{p \text{ odd}} r_p, \sum_{p \text{ even}} r_p)$ acts transitively on $\cD$, with isotropy group $\mathrm{V} = \prod_p \mathrm{U}( h_p) \simeq \prod_p \mathrm{U}(r_p)$. By associating to any decomposition $E = \bigoplus_{p \in \bZ} E^p$ the corresponding filtration $ F^p := \bigoplus_{ q \geq p} E^q$ of $V$, one identifies $\cD$ with an open subset (for the Euclidean topology) of a flag variety. This provides $\cD$ with a canonical structure of complex homogeneous manifold. The data of a polarized complex variation of Hodge structure with ranks $r_p$ on $X$ is then equivalent to the data of its monodromy representation $\pi_1 (X, x_0) \arrow \mathrm{G}$ along with the associated period map, which is a $\pi_1 (X, x_0)$-equivariant holomorphic map $\tilde{p} : \tilde{X} \arrow \cD$ from the universal cover $\tilde{X}$ of $X$ which is tangent to the so-called “horizontal” subbundle of the tangent bundle of $\cD$ (this last condition corresponds to Griffiths' transversality).\\

The period domain $\cD $ admits a canonical $\mathrm{G}$-invariant hermitian metric. Griffiths and Schmid proved that its holomorphic sectional curvatures corresponding to horizontal directions are at most a negative constant \cite[Theorem 9.1]{Griffiths-Schmid}, see also \cite[Theorem 13.3.3]{CMSP}. Since the period map $\tilde{p} : \tilde{X} \arrow \cD$ is always tangent to the “horizontal” subbundle of $T_{\cD}$ and the holomorphic sectional curvature
decreases on submanifolds,  one gets by pulling-back the metric on $\cD$ a canonical hermitian metric $g$ with negative holomorphic sectional curvature on the Zariski-open subset $U \subset X$ where $\tilde{p}$ is an immersion. Peters showed that $g$ has nonpositive holomorphic bisectional curvature on $U$ \cite[Corollary 1.8, Lemma 3.1]{Peters90}. Finally, $g$ is a K\"ahler metric on $U$ (even though the metric on $\cD$ is only a hermitian metric) \cite[Theorem 1.2]{Lu99}.

\subsection{Preliminaries on singular hermitian metrics on vector bundles}

We recall here very briefly the basic definitions concerning singular hermitian metrics of semi-negative curvature on vector bundles, after \cite{Berndtsson-Paun, Paun16, Paun-Takayama, Hacon-Popa-Schnell}.

\begin{defn}
A singular hermitian inner product on a finite-dimensional complex
vector space $V$ is a function $\Vert \cdot \Vert : V  \arrow [0, + \infty ]$ with the following properties:
\begin{enumerate}
\item $\Vert \lambda \cdot v \Vert = |\lambda| \cdot \Vert v \Vert$ for every $\lambda \in \bC \backslash \{0\}$ and every $v \in V$, and $\Vert 0 \Vert = 0$.
\item $\Vert v + w \Vert  \leq \Vert v \Vert + \Vert w \Vert$ for every $v, w \in V$, where by convention $\infty \leq \infty$.
\item $\Vert v + w \Vert^2 + \Vert v - w \Vert^2 = 2 \cdot \Vert v \Vert^2 + 2 \cdot \Vert w \Vert^2$ for every $v, w \in V$.
\end{enumerate}
\end{defn}

Let $V_0$ (resp. $V_{\mathrm{fin}}$) be the subset of $V$ of vectors with zero (resp. finite) norm. It follows easily from the axioms that both $V_0$ and $V_{\mathrm{fin}}$ are linear subspaces of $V$. By definition, $\Vert \cdot \Vert$ is said positive definite if $V_0 = 0$ and finite if $V_{\mathrm{fin}} = V$.

\begin{defn}(compare \cite[Definition 17.1]{Hacon-Popa-Schnell} and \cite[Definition 2.8]{Paun16})
Let $X$ be a connected complex manifold and $\cA$ be a non-zero holomorphic vector bundle on $X$. A singular hermitian metric on $\cA$ of semi-negative curvature is a function $h$ that associates to every point $x \in X$ a singular hermitian inner product $\Vert \cdot \Vert_{h,x} : \cA_x \arrow [0, + \infty ]$ on
the complex vector space $\cA_x$, subject to the following two conditions:
\begin{itemize}
\item $h$ is finite and positive definite almost everywhere, meaning that for all $x$
outside a set of measure zero, $\Vert \cdot \Vert_{h,x}$ is a hermitian inner product on $\cA_x$.
\item the function $x \mapsto \log  \Vert s(x) \Vert_{h,x}$ is plurisubharmonic whenever $U \subset X$ is open and $s \in \HH^0(U, \cA)$.
\end{itemize}
\end{defn}

Note that in particular the singular hermitian inner product $\Vert \cdot \Vert_{h,x}$ is finite at each point of $X$. Conversely, we have the:

\begin{lem}\label{extension of metric with semi-negative curvature} Let $\cA$ be a non-zero holomorphic vector bundle on a connected complex manifold $X$. Let $U$ be an open dense subset of $X$, and $h$ be a singular Hermitian metric with semi-negative curvature on $\cA_{|U}$. The following two assertions are equivalent:
\begin{enumerate}
\item $h$ is a restriction to $U$ of a singular Hermitian metric with semi-negative curvature on $\cA$,
\item the function $x \mapsto   \Vert s(x) \Vert_{h,x}$ is locally bounded in the neighborhood of every point of $X$ for every local section $s \in \HH^0(U, \cA)$.
\end{enumerate}
Moreover, the metric extending $h$ is unique when it exists.
\end{lem}
\begin{proof}
This is a direct application of Riemann extension theorem for plurisubharmonic functions.
\end{proof}

\subsection{Proof of Theorem \ref{metric_criterion}}

\begin{prop} \label{extension_of_the_metric} Let $X$ be a complex manifold, $D \subset X$ be a normal crossing divisor and $h$ be a singular hermitian metric with semi-negative curvature on the tangent bundle $T_X|_{X-D}$ restricted to $X-D$. Assume moreover that in restriction to a dense Zariski-open subset the metric $h$ is smooth and has negative holomorphic sectional curvature, bounded from above by a negative constant $-A$. Then $h$ extends uniquely to a singular hermitian metric with semi-negative curvature on the log-tangent bundle $T_X(- \log D)$. 
\end{prop}

\begin{proof}
In view of Lemma \ref{extension of metric with semi-negative curvature}, one has to prove that for any analytic open subset $U$ of $X$ and any local section $t$ of $T_X(- \log D)$ on $U$, the function 
\[ \Vert t_{|U-D} \Vert_h : U -D \arrow [0, + \infty ), \, x \mapsto \Vert t_{|U-D} (x) \Vert_h \]
is locally bounded. In particular, it is sufficient to show that for $X = \Delta^n$, $D =  \cup_{i=1}^r  \{\underline{z} =(z_1, \cdots, z_n) \in \Delta^n \, | \, z_i = 0 \}$ and $t$ a global section of $T_X(- \log D)$, the function $ x \mapsto \Vert t_{|U-D} (x) \Vert_h $ is bounded in a neighborhood of the origin. Writing $t = \sum_{i=1}^r a_i \cdot z_i  \frac{\partial}{\partial z_i} + \sum_{i=r+1}^n a_i  \cdot \frac{\partial}{\partial z_i} $, where the $a_i$ are holomorphic functions on $\Delta^n$, we see that it is sufficient to prove that $  \Vert  z_i \frac{\partial}{\partial z_i} \Vert_h$ (resp. $  \Vert  \frac{\partial}{\partial z_i} \Vert_h$) are bounded in a neighborhood of the origin for $1 \leq i \leq r$ (resp. for $r+1 \leq i \leq n$).
In the first case, since the holomorphic sectional curvature decreases in submanifold, the Ahlfors-Schwarz Lemma (see below) applied to the punctured unit disk passing through $z$ and directed by $ \frac{\partial}{\partial z_i}(z)$ gives that 

\[    \Vert  z_i \frac{\partial}{\partial z_i} \Vert_h  \leq \frac{2}{A} \cdot \frac{1}{- \log |z_i|} . \]

Similarly, in the second case one gets 

\[    \Vert   \frac{\partial}{\partial z_i} \Vert_h  \leq \frac{2}{A} \cdot \frac{1}{(1- |z_i|^2)} . \]

This finishes the proof.
\end{proof}

\begin{lem}[Ahlfors-Schwarz Lemma, cf. {\cite[lemma 3.2]{Demailly_Santa_Cruz}}]
Let $h(t) = h_0(t)i dt \wedge d \bar t$ be a singular hermitian metric on the unit disk $\Delta$, where $\log h_0$ is a subharmonic function such that $i \partial \bar \partial \log h_0(t) \geq A  \cdot h(t)$ in
the sense of currents, for some positive constant $A$. Then $h$ can be compared with the Poincar\'e metric of $\Delta$ as follows:
\[ h(t) \leq \frac{2}{A} \cdot \frac{|dt|^2}{(1- |t|^2)^2}. \]

Similarly, if $h$ is a singular hermitian metric on the punctured unit disk $\Delta^\ast$ satisfying the same assumptions as above, then the following inequality holds:
\[ h(t) \leq \frac{2}{A} \cdot \frac{|dt|^2}{|t|^2(\log |t|)^2}. \]
 
\end{lem}

\textit{End of the proof of Theorem \ref{metric_criterion}.}
We keep the notations of the statement. Thanks to Proposition \ref{extension_of_the_metric}, the log-tangent bundle $T_X(- \log D)$ is endowed with a singular hermitian metric $h$ with semi-negative curvature. This metric induces a singular hermitian metric with semi-negative curvature on the line bundle $\det T_X(- \log D)$. Denoting by $\Theta(h)$ the curvature of $h$ at a point of $X$ where the metric $h$ is smooth, the curvature of the induced metric $\det h$ on $\det T_X(- \log D)$ is given by $\mathrm{tr}(\Theta(h))$. For any nonzero $v \in T^{1,0}_x X$, that we complete in a $h_x$-orthogonal basis $v = u_1, \cdots , u_n$ of $ T^{1,0}_x X$, we have 

\[   \mathrm{tr}(\Theta(h)) (v, \bar v) =  \sum_{i = 1}^n h( \Theta(h)(v , \bar v)(u_i),  u_i ) .\] 

Since by our assumptions $h$ has semi-negative curvature, all the terms in the sum are non-positive. But $h$ is also assumed to have negative holomorphic sectional curvature at $x$, hence the first term $h( \Theta(h)(v)(v , \bar v),  v )$ is negative and the sum is negative. Using \cite[Theorem 1.3]{Popovici08} (which generalizes \cite[Theorem 1.2]{Boucksom} to the non-necessarily K\"ahler manifolds), one gets that the line bundle $\omega(D) = (\det  T_X(- \log D))^{\vee}$ is big.
In particular, $X$ is Moishezon. By taking a modification, one can assume that $X$ is projective (this does not change the conclusions of the Theorem). Since the log-tangent bundle $T_X(- \log D)$ admits a singular hermitian metric $h$ with semi-negative curvature, it follows from a theorem of P{\u a}un and Takayama (cf. {\cite[Theorem 2.21]{Paun16}} and {\cite[Theorem 2.5.2]{Paun-Takayama}}) that the logarithmic cotangent bundle is weakly-positive.

Finally, we prove that $\Omega^1_X(\log D)$ is big if we assume that $h$ is K\"ahler in the neighborhood of a point $x \in X-D$ where it is smooth. The argument is completely similar to what has been done in \cite{Cadorel}, so we will only sketch it briefly. Denote by $\hat{h}$ the metric induced by $h$ on the tautological line bundle $\mathcal O(1) \longrightarrow \mathbb P\left(T_X(- \log D)\right)$. Since $h$ has semi-negative curvature, $\hat{h}$ has positive curvature in the sense of currents, and \cite[Theorem 1.3]{Popovici08} (or \cite[Theorem 1.2]{Boucksom} in the K\"ahler case) implies that 
\begin{equation} \label{ineqvolume}
\mathrm{vol}(\mathcal O(1)) \geq \left( \frac{i}{2 \pi} \right)^{2n-1} \int_{\mathbb P(T_X( - \log D))} {\Theta(\hat{h})}_{ac}^{2n -1}
\end{equation}
where $\Theta(\hat{h})_{ac}$ is the absolutely continuous part of the positive current $\Theta(\hat{h})$. Since $\hat{h}$ is smooth on a Zariski open subset of $\mathbb P(T_X(- \log D)$, the integral in \eqref{ineqvolume} is equal to the integral of $\Theta(\hat{h})^{2n-1}$ on the open set where $\hat{h}$ is smooth. Moreover, since $h$ is K\"ahler in the neighborhood of some point $x \in X-D$, it follows from \cite[Lemma 1.4]{BKT13} that there exists $v \in T_{X,x}$ such that $\Theta(\hat{h})_{(x, [v])}$ is a positive $(1,1)$-form. It follows that the right hand side of \eqref{ineqvolume} is positive, so $\Omega^1_X (\log D)$ is big.

\section{Appendix}

\begin{lem} \label{det_nef_and_big}
If $E$ is a big and weakly-positive vector bundle on a projective variety $X$, then its determinant line bundle $\det E$ is big.
\end{lem}
\begin{proof}
Fix $A$ an ample line bundle on $X$. Since $E$ is big, there exists an integer $k \geq 1$ and an injective map of $\cO_X$-modules $ 0 \to A \to \Sym^k E$. This can be completed in an exact sequence of $\cO_X$-modules $ 0 \to A \to \Sym^k E \to Q$. Observe that $Q$, being a quotient of a symmetric power of weakly-positive vector bundle, is weakly-positive too \cite[Lemma 1.4]{Viehweg83}. The previous exact sequence implies that $\det (\Sym^k E) \simeq A \otimes \det Q$. The right hand-side is a big line bundle and the left-hand side can be expressed as a positive tensor power of $\det E$, hence $\det E$ is big.
\end{proof}

\bibliographystyle{alpha}
\bibliography{biblio}

\vspace{0.5cm}

\textsc{Yohan Brunebarbe, Institut f\"ur Mathematik, Universit\"at Z\"urich, Winterthurerstrasse 190, CH-8057 Z\"urich, Schweiz} \par\nopagebreak
  \textit{E-mail address}: \texttt{yohan.brunebarbe@math.uzh.ch}

\vspace{0.5cm}

\textsc{Beno\^{i}t~Cadorel, Aix Marseille Universit\'{e}, CNRS, Centrale Marseille, I2M, UMR~7373, 13453 Marseille, France} \par\nopagebreak
  \textit{E-mail address}: \texttt{benoit.cadorel@univ-amu.fr}
\vfill

\end{document}